\documentclass[12pt,a4paper]{article}
\usepackage[utf8x]{inputenc}
\usepackage{amsmath}
\usepackage{amsfonts}
\usepackage{amssymb}
\usepackage{amsthm}
\usepackage{mathrsfs}
\usepackage{fancyhdr}
\usepackage{graphicx}
\usepackage[usenames,x11names]{xcolor}
\usepackage{arabtex}
\usepackage{framed}
\usepackage[top=2cm,bottom=2cm,margin=2cm]{geometry}
\usepackage{cancel}
\usepackage{array}   

\usepackage[dvipdfm,pdfstartview=FitH,colorlinks=true,linkcolor=Red4,urlcolor=Red4,%
filecolor=green,citecolor=red]{hyperref}

\title{A new generalization of the Genocchi numbers and its consequence on the Bernoulli polynomials}
\author{\sc Bakir FARHI \\
Laboratoire de Mathématiques appliquées \\
Faculté des Sciences Exactes \\
Université de Bejaia, 06000 Bejaia, Algeria \\[1mm]
\href{mailto:bakir.farhi@gmail.com}{bakir.farhi@gmail.com} \\[1mm]
\url{http://farhi.bakir.free.fr/}
}
\date{}

\let\up=\textsuperscript

\def\R{{\mathbb R}}
\def\Q{{\mathbb Q}}

\def\N{{\mathbb N}}
\def\Z{{\mathbb Z}}

\def\G{\mathcal{G}}
\def\int{\mathrm{Int}}

\def\den{\mathrm{den}}    

\def\EMdash{\leavevmode\hbox to 10.6mm{\vrule height .63ex depth -.59ex
    width 10mm\hfill}}

\theoremstyle{plain}
\numberwithin{equation}{section}
\newtheorem{thm}{Theorem}[section]

\newtheorem{lemma}[thm]{Lemma}

\newtheorem{rmq}[thm]{Remark}
\newtheorem{prop}[thm]{Proposition}
\newtheorem{coll}[thm]{Corollary}

\begin{document}
\maketitle
\begin{abstract}
This paper presents a new generalization of the Genocchi numbers and the Genocchi theorem. As consequences, we obtain some important families of integer-valued polynomials those are closely related to the Bernoulli polynomials. Denoting by ${(B_n)}_{n \in \mathbb{N}}$ the sequence of the Bernoulli numbers and by ${(B_n(X))}_{n \in \mathbb{N}}$ the sequence of the Bernoulli polynomials, we especially obtain that for any natural number $n$, the reciprocal polynomial of the polynomial $\big(B_n(X) - B_n\big)$ is integer-valued. 
\end{abstract}

\section{Introduction and Notations}
Throughout this paper, we let $\N^*$ denote the set of positive integers. For a given prime number $p$, we let $\vartheta_p$ denote the usual $p$-adic valuation. The rational numbers $x$ satisfying $\vartheta_p(x) \geq 0$ are called {\it $p$-integers}; they constitute a subring of $\Q$, usually denoted by $\Z_{(p)}$. For a given natural number $r$, we let $\den(r)$ denote the denominator of $r$; that is the smallest positive integer $d$ such that $d r \in \Z$.

Next, we let $\Q[X]$ denote the ring of polynomials in $X$ with coefficients in $\Q$. If $P \in \Q[X]$, we let $\deg P$ denote the degree of $P$. We call the \textit{reciprocal polynomial} of a polynomial $P \in \Q[X]$ the polynomial $P^*$ ($\in \Q[X]$) obtained by reversing the order of the coefficients of $P$; for example $(2 X^3 + 5 X^2 + 7 X + 3)^* = 3 X^3 + 7 X^2 + 5 X + 2$. It is easy to show that for any $P \in \Q[X]$, we have $P^*(X) = X^{\deg P} P(\frac{1}{X})$. We let $\Delta$ denote the forward difference operator on $\Q[X]$; that is $(\Delta P)(X) := P(X + 1) - P(X)$ ($\forall P \in \Q[X]$). A polynomial $P \in \Q[X]$ whose value $P(n)$ is an integer for every integer $n$ (i.e., $P(\Z) \subset \Z$) is called an \textit{integer-valued polynomial}. The set of integer-valued polynomials is denoted by $\int(Z)$ and forms a $\Z$-algebra (under the usual operations on polynomials). It is known (see, e.g., \cite{cah,pol}) that $\int(Z)$ (seen as a $\Z$-module) is free with infinite rank and has as a basis the sequences of polynomials $\binom{X}{n}$ ($n \in \N$), where $\binom{X}{n} := \frac{X (X - 1) \cdots (X - n + 1)}{n!}$ ($\forall n \in \N$). An exhaustive study of the integer-valued polynomials (including the integer-valued polynomials on a general domain) is given in the book of Cahen and Chabert \cite{cah}.  

Further, \textit{the Bernoulli polynomials} $B_n(X)$ ($n \in \N$) can be defined by the exponential generating function:
$$
\frac{t e^{X t}}{e^t - 1} ~=~ \sum_{n = 0}^{+ \infty} B_n(X) \frac{t^n}{n!}
$$
and \textit{the Bernoulli numbers} $B_n$ are the values of the Bernoulli polynomials at $X = 0$; that is $B_n := B_n(0)$ ($\forall n \in \N$). To mark the difference between the Bernoulli polynomials and the Bernoulli numbers, we always put the indeterminate $X$ in evidence when it comes to polynomials. The Bernoulli polynomials and numbers have many important and remarkable properties; an elementary presentation (but quite rich) can be found in the book of Nielsen \cite{nie}. It is known for example that $\deg B_n(X) = n$ ($\forall n \in \N$) and that $B_n = 0$ for any odd integer $n \geq 3$.

Throughout this paper, we deal with \textit{formal power series} with rational coefficients. We denote by $\Q[[t]]$ the ring of formal power series in $t$ with coefficients in $\Q$. An element $S$ of $\Q[[t]]$ is always represented as
$$
S(t) ~:=~ \sum_{n = 0}^{+ \infty} a_n \frac{t^n}{n!} ,
$$
where $a_n \in \Q$ ($\forall n \in \N$). The $a_n$'s are called \textit{the differential coefficients} of $S$ (because it is immediate that each $a_n$ is the $n$\up{th} derivative of $S$ at $0$). If the $a_n$'s are all integers, we say that $S$ is an \textit{IDC-series} (IDC abbreviates the expression ``with Integral Differential Coefficients''). Many usual functions are IDC-series; we can cite for example the functions $x \mapsto e^x$, $x \mapsto \sin{x}$, $x \mapsto \cos{x}$, $x \mapsto \ln(1 + x)$, and so on. The sum of two IDC-series is obviously an IDC-series. The product of two IDC-series is also an IDC-series. Indeed, if $S_1$ and $S_2$ are two IDC-series with
$$
S_1(t) ~=~ \sum_{n = 0}^{+ \infty} a_n \frac{t^n}{n!} ~~,~~ S_2(t) ~=~ \sum_{n = 0}^{+ \infty} b_n \frac{t^n}{n!}
$$
(so $a_n , b_n \in \Z$, $\forall n \in \N$) then we have
\begin{multline*}
S_1(t) S_2(t) ~=~ \left(\sum_{k = 0}^{+ \infty} a_k \frac{t^k}{k!}\right) \left(\sum_{\ell = 0}^{+ \infty} b_{\ell} \frac{t^{\ell}}{\ell!}\right) ~=~ \sum_{k , \ell \in \N} \frac{(k + \ell)!}{k! \ell!} a_k b_{\ell} \frac{t^{k + \ell}}{(k + \ell)!} \\
=~ \sum_{n = 0}^{+ \infty} \left(\sum_{\begin{subarray}{c}
k , \ell \in \N \\
k + \ell = n
\end{subarray}} \frac{n!}{k! \ell!} a_k b_{\ell}\right) \frac{t^n}{n!} ~=~ \sum_{n = 0}^{+ \infty} \left(\sum_{k = 0}^{n} \binom{n}{k} a_k b_{n - k}\right) \frac{t^n}{n!} ~=~ \sum_{n = 0}^{+ \infty} c_n \frac{t^n}{n!} ,
\end{multline*}
where
$$
c_n ~:=~ \sum_{k = 0}^{n} \binom{n}{k} a_k b_{n - k} ~~~~~~~~~~ (\forall n \in \N) .
$$
Since $a_k , b_k \in \Z$ ($\forall n \in \N$) then $c_n \in \Z$ ($\forall n \in \N$), showing that $S_1 S_2$ is an IDC-series.

Showing that a given function is an IDC-series is not always easy. The more famous example is perhaps the function $t \mapsto \frac{2 t}{e^t + 1}$ whose expansion into a power series is
$$
\frac{2 t}{e^t + 1} ~=~ \sum_{n = 0}^{+ \infty} G_n \frac{t^n}{n!} ,
$$
where the $G_n$'s are called \textit{the Genocchi numbers} and have been studied by several authors (see, e.g., \cite{com,dum2,gan,rio}). An important theorem of Genocchi \cite{gen} states that the $G_n$'s are all integers; equivalently, the function $t \mapsto \frac{2 t}{e^t + 1}$ is and IDC-series. A familiar proof of this curious result uses the expression of the $G_n$'s in terms of the Bernoulli numbers:
$$
G_n ~=~ - 2 (2^n - 1) B_n ~~~~~~~~~~ (\forall n \in \N) ,
$$
together with the von Staudt-Clausen theorem and the Fermat little theorem.

In this paper, we generalize the Genocchi numbers by considering for a given integer $a \geq 2$, the function $t \mapsto \frac{a t}{e^{(a - 1) t} + e^{(a - 2) t} + \dots + e^t + 1}$ and its expansion into a power series:
$$
\frac{a t}{e^{(a - 1) t} + e^{(a - 2) t} + \dots + e^t + 1} ~=~ \sum_{n = 0}^{+ \infty} G_{n , a} \frac{t^n}{n!} .
$$
For $a = 2$, we simply obtain the usual Genocchi numbers; that is $G_{n , 2} = G_n$ ($\forall n \in \N$). In our main Theorem \ref{t1}, we prove that the $G_{n , a}$'s are all integers, which generalizes the Genocchi theorem. Next, by interpolating the numbers $G_{n , a}$ ($a \geq 2$), for a fixed $n \in \N$, we derive some important families of integer-valued polynomials which are closely related to the Bernoulli polynomials. We particularly obtain that for any natural number $n$, the polynomial $(B_n(X) - B_n)^*$ is integer-valued.

\section{Some preliminaries on the IDC-series}

In this section, we present some selected elementary properties of the IDC-series. We begin with the following proposition:

\begin{prop}\label{p1}
Let $f$ be an IDC-series. Then $\frac{1}{f}$ is an IDC-series if and only if $f(0) = \pm 1$.
\end{prop}

\begin{proof}
Write
$$
f(t) ~=~ \sum_{n = 0}^{+ \infty} a_n \frac{t^n}{n!} ,
$$
where $a_n \in \Z$, $\forall n \in \N$. \\
\textbullet{} If $\frac{1}{f}$ is an IDC-series then we have $(\frac{1}{f})(0) = \frac{1}{f(0)} \in \Z$, which is possible if and only if $f(0) = \pm 1$ (since $f(0) = a_0 \in \Z$). \\
\textbullet{} Conversely, suppose that $f(0) = \pm 1$ (that is $a_0 = \pm 1$) and let us show that $\frac{1}{f}$ is an IDC-series. The fact that $f \in \Q[[t]]$ and $f(0) \neq 0$ implies that $\frac{1}{f} \in \Q[[t]]$; so let
$$
\frac{1}{f(t)} ~=~ \sum_{n = 0}^{+ \infty} b_n \frac{t^n}{n!} ,
$$
where $b_n \in \Q$, $\forall n \in \N$. Thus, we have
$$
\left(\sum_{n = 0}^{+ \infty} a_n \frac{t^n}{n!}\right) \left(\sum_{n = 0}^{+ \infty} b_n \frac{t^n}{n!}\right) ~=~ 1 .
$$
Then, by identifying the differential coefficients in both power series of the last identity, we obtain that:
$$
\left\{\begin{array}{l}
a_0 b_0 ~=~ 1 \\[3mm]
\displaystyle\sum_{k = 0}^{n} \binom{n}{k} a_k b_{n - k} ~=~ 0 ~~~~ (\forall n \geq 1)
\end{array}
\right. .
$$
Hence
$$
\left\{\begin{array}{l}
b_0 ~=~ \frac{1}{a_0} \\[3mm]
b_n ~=~ - \frac{1}{a_0} \left[\binom{n}{1} b_{n - 1} a_1 + \binom{n}{2} b_{n - 2} a_2 + \dots + \binom{n}{n} b_0 a_n\right] ~~~~ (\forall n \geq 1)
\end{array}
\right. ,
$$
showing that $b_0 \in \Z$ (since $a_0 = \pm 1$ by hypothesis) and then (by a simple induction on $n$) that $b_n \in \Z$ for all $n \in \N$. Thus $\frac{1}{f}$ is an IDC-series, as required. Our proof is complete.
\end{proof}

From Proposition \ref{p1}, we derive the following corollary:

\begin{coll}\label{c1}
Let $f(t) = \sum_{n = 0}^{+ \infty} a_n \frac{t^n}{n!}$ be an IDC-series with $a_0 \neq 0$. Then the formal power series $\frac{a_0}{f(a_0 t)}$ is also an IDC-series.
\end{coll}

\begin{proof}
We have
$$
\frac{f(a_0 t)}{a_0} ~=~ \frac{1}{a_0} \sum_{n = 0}^{+ \infty} a_n \frac{(a_0 t)^n}{n!} ~=~ 1 + \sum_{n = 1}^{+ \infty} a_n a_0^{n - 1} \frac{t^n}{n!} ,
$$
showing that $\frac{f(a_0 t)}{a_0}$ is an IDC-series with the first coefficient equal to $1$. According to Proposition \ref{p1}, it follows that $\frac{1}{f(a_0 t) / a_0} = \frac{a_0}{f(a_0 t)}$ is also an IDC-series. This achieves the proof.
\end{proof}

Finally, from Corollary \ref{c1}, we derive the following corollary which is essential for our purpose.

\begin{coll}\label{c2}
Let $f(t) = \sum_{n = 0}^{+ \infty} a_n \frac{t^n}{n!}$ be an IDC-series with $a_0 \neq 0$ and let $\frac{1}{f(t)} = \sum_{n = 0}^{+ \infty} b_n \frac{t^n}{n!} \in \Q[[t]]$ be the reciprocal of the formal power series $f$. Then, for all $n \in \N$, the denominator of the rational number $b_n$ divides the integer $a_0^{n + 1}$.
\end{coll}

\begin{proof}
From $\frac{1}{f(t)} = \sum_{n = 0}^{+ \infty} b_n \frac{t^n}{n!}$, we derive that:
$$
\frac{a_0}{f(a_0 t)} ~=~ a_0 \sum_{n = 0}^{+ \infty} b_n \frac{(a_0 t)^n}{n!} ~=~ \sum_{n = 0}^{+ \infty} b_n a_0^{n + 1} \frac{t^n}{n!} .
$$
But, according to Corollary \ref{c1}, we know that $\frac{a_0}{f(a_0 t)}$ is an IDC-series; equivalently, we have that $b_n a_0^{n + 1} \in \Z$ ($\forall n \in \N$). Consequently, the denominator of each of the rational numbers $b_n$ ($n \in \N$) is a divisor of $a_0^{n + 1}$, as required. The corollary is proved.
\end{proof}

\section{The main result}

Our main result is the following:

\begin{thm}\label{t1}
Let $a \geq 2$ be an integer. Then for any positive integer $n$, the number $G_{n , a}$ is an integer.
\end{thm}

If we take $a = 2$ in Theorem \ref{t1}, we obtain the Genocchi original theorem.

The method of proving Theorem \ref{t1} consists to show that for any prime number $p$, we have $\vartheta_p(G_{n , a}) \geq 0$ ($a \geq 2$, $n \in \N^*$). To do so, we distinguish two cases according as $p$ does or does not divide $a$. We begin with the following proposition:

\begin{prop}\label{p2}
Let $a$ and $n$ be two positive integers with $a \geq 2$. Then the denominator of the rational number $G_{n , a}$ divides $a^{n - 1}$.
\end{prop}

\begin{proof}
By applying Corollary \ref{c2} for the IDC-series
$$
f(t) ~:=~ e^{(a - 1) t} + e^{(a - 2) t} + \dots + e^t + 1 ~=~ a + \sum_{n = 1}^{+ \infty} \left(1^n + 2^n + \dots + (a - 1)^n\right) \frac{t^n}{n!} ,
$$
we obtain that the expansion of $\frac{1}{f}$ into a power series has the form
\begin{equation}\label{eq1}
\frac{1}{f(t)} ~=~ \sum_{n = 0}^{+ \infty} b_n \frac{t^n}{n!} ,
\end{equation}
where, for all $n \in \N$, we have $b_n \in \Q$ and $\den(b_n) \mid a^{n + 1}$. Then, by multiplying the two sides of \eqref{eq1} by $a t$, we get
$$
\frac{a t}{f(t)} ~=~ \sum_{n = 0}^{+ \infty} a b_n \frac{t^{n + 1}}{n!} ~=~ \sum_{n = 1}^{+ \infty} a b_{n - 1} \frac{t^n}{(n - 1)!} ~=~ \sum_{n = 1}^{+ \infty} a n b_{n - 1} \frac{t^n}{n!} .
$$
But since we have on the other hand $\frac{a t}{f(t)} = \sum_{n = 0}^{+ \infty} G_{n , a} \frac{t^n}{n!}$, we deduce that
$$
G_{n , a} ~=~ a n b_{n - 1} ~~~~~~~~~~ (\forall n \in \N^*) .
$$
Now, for a given $n \in \N^*$, we have that $\den(b_{n - 1}) \mid a^n$; thus $\den(a n b_{n - 1}) \mid a^{n - 1}$; that is $\den(G_{n , a}) \mid a^{n - 1}$. This completes the proof.
\end{proof}

From Proposition \ref{p2}, we immediately derive the following corollary:

\begin{coll}\label{c3}
Let $a$ and $n$ be two positive integers with $a \geq 2$. Then for any prime number $p$ not dividing $a$, we have
\begin{equation}
\vartheta_p\left(G_{n , a}\right) ~\geq~ 0 . \tag*{$\square$}
\end{equation}
\end{coll}

Now, we are going to establish the analog of Corollary \ref{c3} for the prime numbers $p$ dividing the considered number $a$. For this purpose, we first need the following proposition:

\begin{prop}\label{p3}
Let $a \geq 2$ be an integer. Then for all positive integer $n$, we have
$$
G_{n , a} + \sum_{1 \leq k \leq n - 1} \binom{n}{k} \frac{a^k}{k + 1} G_{n - k , a} ~=~ 1 .
$$
\end{prop}

\begin{proof}
From the definition of the numbers $G_{n , a}$, we have
$$
\left(\sum_{n = 0}^{+ \infty} G_{n , a} \frac{t^n}{n!}\right) \left(\sum_{n = 0}^{+ \infty} \frac{a^n}{n + 1} \frac{t^n}{n!}\right) = \frac{a t}{e^{(a - 1) t} + e^{(a - 2) t} + \dots + e^t + 1} \cdot \frac{e^{a t} - 1}{a t} ~=~ e^t - 1 ~=~ \sum_{n = 1}^{+ \infty} \frac{t^n}{n!} ;
$$
that is
\begin{equation}\label{eq2}
\left(\sum_{n = 0}^{+ \infty} G_{n , a} \frac{t^n}{n!}\right) \left(\sum_{n = 0}^{+ \infty} \frac{a^n}{n + 1} \, \frac{t^n}{n!}\right) ~=~ \sum_{n = 1}^{+ \infty} \frac{t^n}{n!} .
\end{equation}
So, for a given $n \in \N^*$, the identification of the $n$\up{th} differential coefficients in the two hand-sides of \eqref{eq2} gives
$$
\sum_{k = 0}^{n} \binom{n}{k} \frac{a^k}{k + 1} G_{n - k , a} ~=~ 1 ,
$$
which is nothing else the required identity (since $G_{0 , a} = 0$).
\end{proof}

Next, we have the following lemma:

\begin{lemma}\label{l1}
Let $a \geq 2$ be an integer. Then for all prime number $p$ dividing $a$ and all natural number $k$, we have
$$
\vartheta_p\left(\frac{a^k}{k + 1}\right) ~\geq~ 0 .
$$
\end{lemma}

\begin{proof}
Let $p$ be a prime number dividing $a$ (so $\vartheta_p(a) \geq 1$) and $k$ be a natural number. Since $k + 1 \leq 2^k \leq p^k$ then we have $\vartheta_p(k + 1) \leq k$. Hence
$$
\vartheta_p\left(\frac{a^k}{k + 1}\right) ~=~ k \vartheta_p(a) - \vartheta_p(k + 1) ~\geq~ k - \vartheta_p(k + 1) ~\geq~ 0 ,
$$
as required.
\end{proof}

From Proposition \ref{p3} and Lemma \ref{l1}, we derive the following corollary:

\begin{coll}\label{c4}
Let $a$ and $n$ be two positive integers with $a \geq 2$. Then for any prime number $p$ dividing $a$, we have
$$
\vartheta_p\left(G_{n , a}\right) ~\geq~ 0 .
$$
\end{coll}

\begin{proof}
Let $p$ be a prime number dividing $a$. To prove that $\vartheta_p(G_{n , a}) \geq 0$, we argue by induction on $n \in \N^*$ and use the identity of Proposition \ref{p3} together with Lemma \ref{l1}. \\
\textbullet{} For $n = 1$, we have $G_{1 , a} = 1$, so $\vartheta_p(G_{1 , a}) = 0 \geq 0$. \\
\textbullet{} Let $n \geq 2$ be an integer. Suppose that $\vartheta_p(G_{m , a}) \geq 0$ for any positive integer $m < n$ and show that $\vartheta_p(G_{n , a}) \geq 0$. By Proposition \ref{p3}, we have
$$
G_{n , a} ~=~ - \sum_{1 \leq k \leq n - 1} \binom{n}{k} \frac{a^k}{k + 1} G_{n - k , a} + 1 .
$$
Since the binomial coefficients are known to be integers, the numbers $G_{n - k , a}$ ($1 \leq k \leq n - 1$) are $p$-integers (by the induction hypothesis) and the numbers $\frac{a^k}{k + 1}$ ($1 \leq k \leq n - 1$) are $p$-integers (by Lemma \ref{l1}) then the sum $- \sum_{1 \leq k \leq n - 1} \binom{n}{k} \frac{a^k}{k + 1} G_{n - k , a} + 1$ is a $p$-integer; that is $\vartheta_p(G_{n , a}) \geq 0$, as required. This achieves the induction, and hence, the proof. 
\end{proof}

The proof of our main result is now immediate:

\begin{proof}[Proof of Theorem \ref{t1}]
Let $a$ and $n$ be two positive integers with $a \geq 2$. According to Corollaries \ref{c3} and \ref{c4}, we have for any prime number $p$: $\vartheta_p(G_{n , a}) \geq 0$. Thus the number $G_{n , a}$ is an integer. Our main result is proved.
\end{proof}

\section{Some consequences of the main result}

For the following, we extend the numbers $G_{n , a}$ to non-integer values of $a$. Precisely, we define $\G_n(x)$ ($n \in \N$, $x \in \R$) as the coefficients occurring on the right-hand side of the identity:
$$
\frac{e^{x t}}{e^{x t} - 1} (e^t - 1) ~=~ \sum_{n = 0}^{+ \infty} \G_n(x) \frac{t^n}{n!} ,
$$
where it is understood that $\G_n(0) = \displaystyle\lim_{x \rightarrow 0} \G_n(x) = \begin{cases}
0 & \text{if } n = 0 \\
1 & \text{otherwise}
\end{cases}$ (because $\displaystyle\lim_{x \rightarrow 0} \frac{e^{x t}}{e^{x t} - 1} (e^t - 1) = e^t - 1$). Then it is immediate that $\G_n(a) = G_{n , a}$ if $n \in \N$ and $a$ is an integer $\geq 2$. The following proposition shows that for any $n \in \N$, the function $x \mapsto \G_n(x)$ is actually a polynomial which depends on the Bernoulli polynomial $B_n(X)$.

\begin{prop}\label{p4}
For all natural number $n$, we have
$$
\G_n(X) ~=~ \left(B_n(X) - B_n\right)^* ~=~ \sum_{k = 0}^{n - 1} \binom{n}{k} B_k X^k .
$$
So $\G_n$ ($n \in \N$) is a polynomial with degree $\leq n - 1$.
\end{prop}

\begin{proof}
By definition of the Bernoulli polynomials, we have
$$
\frac{t e^{X t}}{e^t - 1} ~=~ \sum_{n = 0}^{+ \infty} B_n(X) \frac{t^n}{n!} .
$$
Then, By substituting in the latter $X$ by $\frac{1}{X}$ and $t$ by $X t$, we get
$$
\frac{X t e^t}{e^{X t} - 1} ~=~ \sum_{n = 0}^{+ \infty} X^n B_n\left(\frac{1}{X}\right) \frac{t^n}{n!} ;
$$
that is (since $\deg B_n = n$, $\forall n \in \N$)
\begin{equation}\label{eq3}
\frac{X t}{e^{X t} - 1} e^t ~=~ \sum_{n = 0}^{+ \infty} B_n^*(X) \frac{t^n}{n!} .
\end{equation}
On the other hand, we have by definition of the Bernoulli numbers:
$$
\frac{X t}{e^{X t} - 1} ~=~ \sum_{n = 0}^{+ \infty} B_n \frac{(X t)^n}{n!} ,
$$
that is
\begin{equation}\label{eq4}
\frac{X t}{e^{X t} - 1} ~=~ \sum_{n = 0}^{+ \infty} B_n X^n \frac{t^n}{n!} .
\end{equation}
By subtracting side to side \eqref{eq4} from \eqref{eq3}, we finally obtain
$$
\frac{X t}{e^{X t} - 1} (e^t - 1) ~=~ \sum_{n = 0}^{+ \infty} \big(B_n^*(X) - B_n X^n\big) \frac{t^n}{n!} .
$$
Comparing this with the identity defining the $\G_n(X)$'s, we derive that for all $n \in \N$, we have
$$
\G_n(X) ~=~ B_n^*(X) - B_n X^n ~=~ \left(B_n(X) - B_n\right)^* ,
$$
as required. The second equality of the proposition immediately follows from the well-known expression of the Bernoulli Polynomials in terms of the Bernoulli numbers, which is $B_n(X) = \sum_{k = 0}^{n} \binom{n}{k} B_k X^{n - k}$ ($\forall n \in \N$). This completes the proof.
\end{proof}

\begin{rmq}
Since we know that $B_n = 0$ if and only if $n$ is an odd integer $\geq 3$, then from the formula of Proposition \ref{p4}, we can precise the degree of the polynomial $\G_n$ ($n \in \N^*$). We have that:
$$
\deg \G_n ~=~ \begin{cases}
n - 1 & \text{if } n \text{ is odd or } n = 2 \\
n - 2 & \text{if } n \text{ is even and } n \geq 4 
\end{cases} .
$$
\end{rmq}

Further, from Theorem \ref{t1}, we derive the following corollary:

\begin{coll}\label{c5}
For any natural number $n$, the polynomial $\G_n(X)$ is integer-valued.
\end{coll}

To prove this corollary, we lean on the following well-known lemma (see, e.g., \cite{cah}):

\begin{lemma}\label{l2}
Let $d \in \N$ and $P$ be a polynomial of $\Q[[X]]$ with degree $d$. For $P$ to be an integer-valued polynomial, it suffices that $P$ takes integer values for $(d + 1)$ consecutive integer values of $X$. 
\end{lemma}

\begin{proof}
Suppose that $P$ takes integer values for some $(d + 1)$ consecutive integer values of $X$, which are: $a , a + 1 , \dots , a + d$ ($a \in \Z$) and let us show that $P$ is an integer-valued polynomial. Since $\deg P = d$ then we have that $\Delta^{d + 1} P = 0$; that is
$$
P(X + d + 1) ~=~ \sum_{k = 0}^{d} (-1)^{d - k} \binom{d + 1}{k} P(X + k) .
$$
Using this identity, we immediately deduce by induction that:
\begin{equation}\label{eq5}
P(x) \in \Z ~~~~~~~~~~ (\forall x \in \Z , x \geq a + d) .
\end{equation}
Next, if we take instead of $P(X)$ the polynomial $P(-X)$, which has the same degree with $P$ and takes integer values for the $(d + 1)$ consecutive integer values $- a - d$, $- a - d + 1$, $\dots$, $- a$ of $X$, we similarly obtain that:
$$
P(- x) \in \Z ~~~~~~~~~~ (\forall x \in \Z , x \geq - a) ;
$$
that is
\begin{equation}\label{eq6}
P(x) \in \Z ~~~~~~~~~~ (\forall x \in \Z , x \leq a) .
\end{equation}
From \eqref{eq5} and \eqref{eq6}, we conclude that $P(x) \in \Z$, $\forall x \in \Z$. Thus $P$ is an integer-valued polynomial. The lemma is proved.
\end{proof}

Let us now prove Corollary \ref{c5}:

\begin{proof}[Proof of Corollary \ref{c5}]
Let $n \in \N$ be fixed. Since for any integer $a \geq 2$, we have $\G_n(a) = G_{n , a} \in \Z$ (according to Theorem \ref{t1}) then the polynomial $\G_n(X)$ takes integer values for an infinite number of consecutive integer values of $X$. It follows (according to Lemma \ref{l2}) that $\G_n(X)$ is an integer-valued polynomial. This achieves the proof.
\end{proof}

Next, from Proposition \ref{p4} and Corollary \ref{c5}, we derive the following curious result concerning the Bernoulli polynomials of odd degree.

\begin{coll}\label{c6}
For any odd integer $n \geq 3$, the reciprocal polynomial of the Bernoulli polynomial $B_n(X)$ is integer-valued. 
\end{coll}

\begin{proof}
This is an immediate consequence of Proposition \ref{p4}, Corollary \ref{c5} and the well-known fact that $B_n = 0$ for $n$ odd, $n \geq 3$.
\end{proof}

We now turn to present another important property concerning the reciprocal polynomials of some particular type of integer-valued polynomials. For a given $n \in \N$, we define
$$
\sigma_n(a) ~:=~ 0^n + 1^n + \dots + a^n ~~~~~~~~~~ (\forall a \in \N) ,
$$
where we convention that $0^0 = 1$. \\
It has been known for a very long time that $\sigma_n(a)$ is polynomial on $a$ with degree $(n + 1)$, but a closed form of the polynomial in question was discovered for the first time by Jacob Bernoulli \cite{ber} and it is given by:
$$
\sigma_n(a) ~=~ \frac{1}{n + 1} \sum_{k = 0}^{n} \binom{n + 1}{k} B_k a^{n + 1 - k} + a^n .
$$
For a given $n \in \N$, let us define $\sigma_n(X)$ as the polynomial interpolating the sequence ${(\sigma_n(a))}_{a \in \N}$; that is
\begin{equation}\label{eq7}
\sigma_n(X) ~=~ \frac{1}{n + 1} \sum_{k = 0}^{n} \binom{n + 1}{k} B_k X^{n + 1 - k} + X^n .
\end{equation}
For $n \in \N$, since the $\sigma_n(a)$'s ($a \in \N$) are obviously all integers then (according to Lemma \ref{l2}) the polynomial $\sigma_n(X)$ is an integer-valued polynomial. But what about its reciprocal polynomial? The previous results permit us to obtain something in this direction. We have the following proposition:

\begin{prop}\label{p5}
For any natural number $n$, the polynomial $(n + 1) \sigma_n^*(X)$ is integer-valued.
\end{prop}

\begin{proof}
Let $n \in \N$ be fixed. According to \eqref{eq7}, we have
\begin{eqnarray*}
\sigma_n^*(X) & = & \frac{1}{n + 1} \sum_{k = 0}^{n} \binom{n + 1}{k} B_k X^k + X \\
& = & \frac{1}{n + 1} \G_{n + 1}(X) + X ~~~~~~~~~~ (\text{according to Proposition \ref{p4}}) .
\end{eqnarray*}
Hence
$$
(n + 1) \sigma_n^*(X) ~=~ \G_{n + 1}(X) + (n + 1) X .
$$
Since $\G_{n + 1}(X)$ is integer-valued (according to Corollary \ref{c5}), the last equality shows that also $(n + 1) \sigma_n^*(X)$ is integer-valued (as the sum of two integer-valued polynomials). The proposition is proved.
\end{proof}

\subsection*{Two important open problems}
{\large\bf 1.---} Since the polynomials $\G_n(X)$ are integer-valued (according to Corollary \ref{c5}) then they admit representations as linear combinations, with integer coefficients, of the polynomials $\binom{X}{k}$ ($k \in \N$). Precisely, there exist integers $a_{n , k}$ ($n \in \N^*$, $k \in \N$, $0 \leq k < n$) for which we have
$$
\G_n(X) ~=~ a_{n , 0} \binom{X}{0} + a_{n , 1} \binom{X}{1} + \dots + a_{n , n - 1} \binom{X}{n - 1} ~~~~~~~~~~ (\forall n \in \N^*) .
$$
The $a_{n , k}$'s can be calculated for example by using the Newton interpolation formula:
$$
P(X) ~=~ \sum_{k = 0}^{\deg P} \left(\Delta^k P\right)(0) \binom{X}{k} ~~~~~~~~~~ (\forall P \in \Q[X]) .
$$
So, we have that:
$$
a_{n , k} ~=~ \left(\Delta^k \G_n\right)(0) ~~~~~~~~~~ (\forall n \in \N^* , \forall k \in \N , 0 \leq k < n) .
$$
If we arrange those integers $a_{n , k}$ ($0 \leq k < n$) in a triangle array in which each $a_{n , k}$ is the entry in the $n$\up{th} row and $k$\up{th} column, we obtain (after calculation) the following configuration: \pagebreak

\begin{table}[!h]
\begin{center}
\begin{tabular}{l|c@{\hspace*{9mm}}c@{\hspace*{9mm}}c@{\hspace*{9mm}}c@{\hspace*{9mm}}c@{\hspace*{9mm}}c@{\hspace*{9mm}}c}
$n = 1$ & $1$ & ~ & ~ & ~ & ~ & ~ & ~ \\
$n = 2$ & $1$ & $- 1$ & ~ & ~ & ~ & ~ & ~ \\
$n = 3$ & $1$ & $- 1$ & $1$ & ~ & ~ & ~ & ~ \\
$n = 4$ & $1$ & $- 1$ & $2$ & ~ & ~ & ~ & ~ \\
$n = 5$ & $1$ & $- 1$ & $1$ & $- 6$ & $- 4$ & ~ & ~ \\
$n = 6$ & $1$ & $- 1$ & $- 2$ & $- 18$ & $- 12$ ~ & ~ \\
$n = 7$ & $1$ & $- 1$ & $1$ & $48$ & $232$ & $300$ & $120$ \\
$n = 8$ & $1$ & $- 1$ & $18$ & $276$ & $984$ & $1200$ & $480$
\end{tabular}
\end{center}
\caption{The triangle of the $a_{n , k}$'s for $0 \leq k < n \leq 8$}
\end{table}

Note that in this triangle, we have omitted the integers $a_{n , n - 1}$ for the even $n$'s (since they are zero). \\
The interesting problem we pose here consists to find a simple and practical rule to construct the above triangle step by step. 

\medskip

\noindent{\large\bf 2.---} For a polynomial $P \in \int(Z)$, we don't have in general $(\deg P) \cdot P^* \in \int(Z)$ (indeed, the polynomial $\binom{X}{3} = \frac{X (X - 1) (X - 2)}{6}$ provides a counterexample); however, the polynomials $\sigma_n(X)$ ($n \in \N$) satisfy this property (according to Proposition \ref{p5}). So, it is interesting to study for which category of integer-valued polynomials, the above property is satisfied.

\end{document}